\documentclass[10pt]{amsart}

\usepackage{enumerate,amsthm,dsfont}
\usepackage{latexsym,amsfonts,amsmath,,amssymb,amscd}

\newtheorem{theorem}{Theorem}[section]
\newtheorem{lemma}[theorem]{Lemma}

\newtheorem{corollary}[theorem]{Corollary}
\newtheorem{problem}[theorem]{Problem}

\theoremstyle{definition}

\newtheorem{example}[theorem]{Example}

\theoremstyle{remark}

\numberwithin{equation}{section}

\begin{document}
\title{Continuity of the cone spectral radius}

%    Remove any unused author tags.

%    author one information
\author{Bas Lemmens}
\address{School of Mathematics, Statistics \& Actuarial Science, Cornwallis Building, 
University of Kent, Canterbury, Kent CT2 7NF, UK}
\curraddr{}
\email{B.Lemmens@kent.ac.uk}
\thanks{}

%    author two information
\author{Roger Nussbaum}
\address{Department of Mathematics, 
Rutgers, The State University Of New Jersey, 
110 Frelinghuysen Road, Piscataway, NJ 08854-8019, USA}
\curraddr{}
\email{nussbaum@math.rutgers.edu}
\thanks{The second  author was partially supported by NSF DMS-0701171}

%    For articles to be published after 1 January 2010, you may use
%    the following version:
%\subjclass[2010]{Primary }


\date{July 2011}

\dedicatory{}

\begin{abstract}

This paper concerns the question whether the cone spectral radius $r_C(f)$ of a continuous compact order-preserving homogenous map $f\colon C\to C$ on a closed cone $C$ in Banach space $X$ depends continuously on the map. Using the fixed point index we show that if there exists $0<a_1<a_2<a_3<\ldots$ {\bf not} in the cone spectrum, $\sigma_C(f)$, and $\lim_{k\to\infty} a_k = r_C(f)$, then the cone spectral radius is continuous. 
An example is presented showing that if such a sequence $(a_k)_k$ does not exist, continuity may fail. We also analyze the cone spectrum of continuous order-preserving homogeneous maps on finite dimensional closed cones. In particular,  we prove that if $C$ is a polyhedral cone with $m$ faces, then $\sigma_C(f)$ contains at most $m-1$ elements, and this upper bound is sharp for each polyhedral cone. Moreover, for each non-polyhedral cone there exist maps whose cone spectrum is infinite. 
\end{abstract}

\maketitle
\section{Introduction}
In their seminal paper \cite{KR} Krein and Rutman showed, besides the now famous Krein-Rutman Theorem, that some of the spectral properties of compact linear maps leaving invariant a closed cone $C$ in Banach space $X$, can be generalized to certain classes of nonlinear maps. In particular, they considered continuous compact maps $f\colon C\to C$ which are homogeneous of degree one and which preserve the partial ordering induced by $C$. 

If $f\colon C\to C$ is a continuous, homogeneous of degree one map which preserves the partial ordering induced by $C$, one can define the {\em Bonsall cone spectral radius} of $f$ by 
\[
\tilde{r}_C(f)=\lim_{k\to\infty} \|f^k\|_C^{1/k}=\inf_{k\geq 1} \|f^k\|_C^{1/k}
\]
where
\[
\|f^k\|_C =\sup\{\|f^k(x)\|\colon x\in C\mbox{ and }\|x\|\leq 1\}.
\]
This definition of the  cone spectral radius  was introduced by  Bonsall  \cite{Bon} in the context of linear maps on cones. An alternative notion of the cone spectral radius was considered by Mallet-Paret and Nussbaum in \cite[Section 2]{MN1}, and is defined as follows. For $x\in C$, let 
\[
\mu(x)=\limsup_{k\to\infty} \|f^k(x)\|^{1/k}
\]
and define $r_C(f) =\sup \{\mu(x)\colon x\in C\}$. In many important cases it  is known that $\tilde{r}_C(f)=r_C(f)$: see \cite[Theorems 2.2 and 2.3]{MN1} and \cite[Theorem 3.4]{MN2}. In particular, if $f^m$ is compact for some integer $m\geq 1$, then $\tilde{r}_C(f)=r_C(f)$. It remains an open question whether the equality holds for general nonlinear maps $f\colon C\to C$ as above. 

If $C$ is a closed cone in a Banach space $X$ and $f\colon C\to C$ is a continuous, compact, homogeneous (degree one) map preserving the partial ordering induced by $C$ and $r_C(f)>0$, then it follows from \cite[Theorem 2.1]{Nu1} that there exists $u\in C$ with $\|u\|=1$ such that 
\[
f(u) =r_C(f) u.
\]
More precisely, if one applies \cite[Theorem 2.1]{Nu1} to the map $x\mapsto (1/t)f(x)$, where $0<t<r_C(f)$, one deduces as a special case that $f$ has an eigenvector of norm one in $C$  with eigenvalue greater than or equal to $t$. A simple limiting argument implies that $f$ has  an eigenvector in $C$ with eigenvalue $r_C(f)$.
More general results concerning the existence of an eigenvector with eigenvalue $r_C(f)$ can be found in \cite{MN1,MN2}.

In this paper we study the continuity of $r_C(f)$. More precisely we discuss the following problem. Suppose that $f\colon C\to C$ and $f_k\colon C\to C$, for $k\geq 1$,  are continuous compact  homogeneous order-preserving maps such that 
\begin{equation}\label{eq:1.1} 
\lim_{k\to\infty} \|f-f_k\|_C =0.
\end{equation}
Under what further conditions on $f$ does the equality, 
\begin{equation}\label{eq:1.2}
\lim_{k\to\infty} r_C(f_k)=r_C(f),
\end{equation}
hold?  
A similar problem was considered by Kakutani \cite[pp.282--283]{Rick}, who gave an example of a bounded linear map on $\ell_2$ whose spectral radius is discontinuous in the $C^*$-algebra of bounded linear maps on $\ell_2$. The problem was further studied in the setting of Banach algebras in  \cite{Bur,Mo1,Mo2,Mur,New}. 

By using the fixed point index, we prove that if there exist $0<a_1< a_2 < a_3<\ldots$ such that 
\[
\lim_{k\to\infty} a_k=r_C(f)
\] and each $a_k$ is {\bf not} in the {\em cone spectrum}, 
\[
\sigma_C(f)=\{\rho\geq 0\colon f(x)=\rho x\mbox{ for some }x\in C\mbox{ with }x\neq 0\},
\]
of $f$, then (\ref{eq:1.2}) holds. In addition, we provide an example of a linear map 
$T\colon C\to C$ on a closed total cone $C$ in an infinite dimensional Banach space $X$, such that $T$ is compact on $C$, but not  compact as a linear map on the whole of $X$, and for which $(0,r_C(T)]=\sigma_C(T)$,  $r_C(T)>0$, and (\ref{eq:1.2}) fails. We will also give an example of a continuous order-preserving homogeneous map $f$ on a finite dimensional closed cone $C$ with $\sigma_C(f)=[0,1]$. So far, however, 
no finite dimensional example is known for which (\ref{eq:1.2}) fails. 
If $C$ is a polyhedral cone with $m$ faces, we will prove that $\sigma_C(f)$ contains at most $m-1$ points, and show that this upper bound is sharp for every polyhedral cone. 
Thus, the cone spectral radius is continuous for maps on polyhedral cones. 
We will also show that on each  finite dimensional closed non-polyhedral cone $C$, there exists a continuous order-preserving homogeneous map $f\colon C\to C$ such that $\sigma_C(f)$ contains a countably infinite number of distinct points.  

\section{A condition on the cone spectrum}
Throughout the exposition $X$ will be a real  Banach space and $C$ will be a closed cone in $X$. Recall  that $C\subseteq X$ is a {\em cone} if $C$ is convex, 
$\lambda C\subseteq C$ for all  $\lambda\geq 0$, and  $C\cap (-C) =\{0\}$. A cone $C$ in $X$ is said to be {\em total } if  $X=\mathrm{cl}(C-C)$. A cone $C$ induces a partial ordering, $\leq$, on $X$ by $x\leq y$ if $y-x\in C$.  A map $f\colon C\to C$ is said to 
be {\em homogeneous} (of degree one) if $\lambda f(x) =f(\lambda x)$ for all $\lambda \geq 0$ and $x\in C$. It is said to be {\em order-preserving } if $f(x)\leq f(y)$
whenever $x\leq y$. Furthermore $f\colon C\to C$ is called {\em compact} if  for each $D\subseteq C$ bounded, $\mathrm{cl}(f(D))$ is compact. 

The following lemma is elementary, but useful in the sequel. 
\begin{lemma}\label{lem:2.1} 
If $S$ is compact subset of a complete metric space $(M,d)$ and $(x_k)_k$ is a sequence in $M$ with 
\begin{equation}\label{eq:2.1} 
\lim_{k\to\infty} d(x_k,S)=0,
\end{equation}
then $(x_k)_k$ has a convergent subsequence with limit in $S$. 
\end{lemma}
\begin{proof}
Recall that $S$ is compact in $(M,d)$ if and only if every sequence in $S$ has a convergent subsequence. Let $a_k = d(x_k,S)$ and let $b_k=a_k+1/k$, so $b_k\to 0$ as $k\to\infty$. From the assumption we know that there exists $s_k\in S$  such that $d(x_k,s_k)<b_k$. As $S$ is compact, the sequence $(s_k)_k$ has a convergent subsequence with limit  $s\in S$. The corresponding subsequence of $(x_k)_k$ also converges to $s$.
\end{proof}
Another useful basic lemma is the following. 
\begin{lemma}\label{lem:2.1b}
If  $f\colon C\to C$  is a continuous homogeneous  order-preserving map on  a closed cone $C$ in $X$, and $av\leq f(v)$ for some $a>0$ and $v\in C\setminus\{0\}$, then 
$a\leq r_C(f)$.
\end{lemma}
\begin{proof}
Assume, for the sake of contradiction, that $a>r_C(f)\geq 0$. 
Select $\epsilon >0$ such that $a>r_C(f)+\epsilon$. A simple induction argument shows that $v\leq (1/a)^kf^k(v)$ for all $k\geq 1$. The definition of $r_C(f)$, however, implies that $\|f^k(v)\|\leq (r_C(f) +\epsilon)^k $ for all $k$ sufficiently large. Since $a>r_C(f)+\epsilon$, we deduce that 
$ (1/a)^k f^k(v)\to 0$ as $k\to\infty$. But $(1/a)^k f^k(v) -v \in C$ for all $k$; so, letting $k\to\infty$ we find that $-v\in C$, which is impossible, as $v\in C\setminus\{0\}$ and $C\cap (-C)=\{0\}$. 
\end{proof}
 
The next lemma shows that the cone spectral radius is  upper-semicontinuous. 
\begin{lemma}\label{lem:2.2} 
Suppose that $C$ is a closed cone in $X$ and $f\colon C\to C$ is a continuous compact homogeneous order-preserving map. If for each $k\geq 1$, $f_k\colon C\to C$ is a continuous compact homogeneous order-preserving map such that 
\[
\lim_{k\to\infty} \|f-f_k\|_C =0,
\]
then $\limsup_{k\to\infty} r_C(f_k)\leq r_C(f)$.
\end{lemma}
\begin{proof}
Let $\lambda=\limsup_{k\to\infty} r_C(f_k)$. Note that the case $\lambda= 0$ is trivial. So, assume that $\lambda >0$ and write $\lambda_k=r_C(f_k)$ for all $k\geq 1$. 
By \cite[Theorem 2.1]{Nu1} and the fact that $r_C(f_k)=\tilde{r}_C(f_k)$, see \cite[Theorems 2.2 and 2.3]{MN1}, there exist $x_k\in C$ with $\|x_k\|=1$ and $f_k(x_k)=\lambda_k x_k$ for all $k\geq 1$. By taking a subsequence we may assume that $\lim_{k\to\infty} \lambda_k=\lambda$. 

Let $\Sigma=\{x\in C\colon \|x\|=1\}$ and note that $\mathrm{cl}(f(\Sigma))$ is compact.  Moreover, as $\lim_{k\to\infty} \|f-f_k\|_C=0$, 
\[
\lim_{k\to\infty} d(f_k(x_k), \mathrm{cl}(f(\Sigma))) =0. 
\]
By Lemma \ref{lem:2.1} $(f_k(x_k))_k$ has a convergent subsequence $(f_{k_i}(x_{k_i}))_i$ with limit $v\in C$. Put $u=v/\lambda$. As $f_{k_i}(x_{k_i})=\lambda_{k_i} x_{k_i}$ and $\lim_{i\to\infty} \lambda_{k_i} =\lambda$, we see that 
\[
\lim_{i\to\infty} \|x_{k_i}- u\|=\lim_{i\to\infty} \| \frac{\lambda_{k_i}}{\lambda} x_{k_i}- u\| = 
\lim_{i\to\infty} \frac{1}{\lambda} \|\lambda_{k_i} x_{k_i} -v\| =0.
\]
As $f$ is continuous and $\lim_{i\to\infty} \|f(x_{k_i}) -f_{k_i}(x_{k_i})\| =0$, we deduce that 
\[
\lim_{i\to\infty} \|f(x_{k_i}) -\lambda_{k_i}x_{k_i}\|=0=\|f(u) -\lambda u\|.
\]
Thus, $f(u)=\lambda u$ and hence $\lambda \leq r_C(f)$ by Lemma \ref{lem:2.1b}. 
\end{proof}

The following theorem gives a condition on $\sigma_C(f)$ that insures that the cone spectral radius is lower-semicontinuous. The proof relies on ideas from topological degree theory. More to the point, it uses basic properties of the fixed point index, $i_W(g,V)$, where $W$ is a closed convex subset of $X$,  $V$ is a relatively open subset of $W$, and $g\colon V\to W$ is compact. A discussion of  the fixed point index in this setting along with further references to the literature can be found in \cite{DG} and in \cite{Nus}. In our case $W$ will be equal to $C$. 
\begin{theorem}\label{thm:2.3} 
Suppose that $C$ is a closed cone in $X$ and $f\colon C\to C$ is a continuous compact homogeneous order-preserving map. If for each $k\geq 1$, $f_k\colon C\to C$ is a continuous compact homogeneous order-preserving map such that 
\[
\lim_{k\to\infty} \|f-f_k\|_C =0,
\]  
then for each $0<\lambda<r_C(f)$ with  $\lambda \not\in \sigma_C(f)$ we have that  
$\liminf_{k\to\infty} r_C(f_k)\geq \lambda$. 
\end{theorem} 
\begin{proof}
For simplicity write $r=r_C(f)$ and $r_k=r_C(f_k)$. Define $g\colon C\to C$ by 
$g(x)=\lambda^{-1}f(x)$ for $x\in C$. Similarly, let $g_k\colon C\to C$ be given by 
$g_k(x)=\lambda^{-1}f_k(x)$ for all $x\in C$. Clearly 
\[
r_C(g) =\lambda^{-1}r>1\mbox{\quad and \quad}r_C(g_k)=\lambda^{-1}r_k.
\]
Write $\rho = \lambda^{-1}r$ and $\rho_k =\lambda^{-1}r_k$. 
We wish to show that $\liminf_{k\to\infty} \rho_k\geq 1$. 

Let $\Sigma=\{x\in C\colon \|x\|=1\}$ and denote $V =\{x\in C\colon \|x\|<1\}$.  Remark that as  $\lambda\not\in\sigma_C(f)$,   
\begin{equation}\label{eq:2.2}
g(x)\neq x\mbox{\quad for all }x\in \Sigma.
\end{equation}
Furthermore there exists $\delta>0$ such that 
\begin{equation}\label{eq:2.3} 
\|g(x) -x\|>\delta \mbox{\quad for all x}\in \Sigma.
\end{equation} 
Indeed, if $(z_k)_k$ in $\Sigma$  and $\lim_{k\to\infty} \|g(z_k) -z_k\|=0$, then it follows from Lemma \ref{lem:2.1} that $(z_k)_k$ has a convergent subsequence with limit $\xi\in \Sigma$. Continuity of $g$ implies that $g(\xi)=\xi$, which is impossible

We claim that for all $k$ sufficiently large,  
\begin{equation}\label{eq:2.4}
(1-t)g(x) +tg_k(x) \neq x\mbox{\quad for all } x\in \Sigma\mbox{ and } 0\leq t\leq 1.
\end{equation}
For the sake of contradiction suppose that there exist a sequence $(x_k)_k$ of distinct points in $\Sigma$ and a sequence $(t_k)_k$ in $[0,1]$ such that $(1-t_k)g(x_k)+t_kg_k(x_k)=x_k$. 
As $\lim_{k\to\infty} \|g-g_k\|_C=0$, 
\[
\lim_{k\to\infty} d(x_k,\mathrm{cl}(g(\Sigma)))=0. 
\] 
Using the compactness of $g$ we see that $(x_k)_k$ has a convergent subsequence $(x_{k_i})_i$ with limit $\zeta\in \Sigma$ by Lemma \ref{lem:2.1}. The continuity of $g$ now implies that 
\[
g(\zeta)=\lim_{i\to\infty} g(x_{k_i})=\lim_{i\to\infty} (1-t_{k_i})g(x_{k_i})+t_{k_i}g_{k_i}(x_{k_i}) =\zeta, 
\]
which contradicts (\ref{eq:2.2}). 

We now utilize the fixed point index as defined in \cite{Nus}. In particular, if we  apply the homotopy property of the fixed point index we find that  
\begin{equation}\label{eq:2.5}
i_C(g,V) = i_C(g_k,V)\mbox{\quad for all $k$ sufficiently large.}
\end{equation}

It was proved in \cite{Nu1} as a special case of Theorem 2.1 that $i_C(g,V)=0$. Using this fact the proof can be completed as follows. From (\ref{eq:2.5}) we deduce that $i_C(g_k,V)=0$ for all $k$ sufficiently large. 
Now suppose that for some large $k$, $r_C(g_k)\leq 1$. By (\ref{eq:2.4}) there exists no fixed point $x\in \Sigma$ of $g_k$; so, $r_C(g_k)<1$. Consider the homotopy $tg_k(x)$ for $0\leq t\leq 1$. By assumption $tg_k(x)\neq x$ for all $0\leq t\leq 1$ and $x\in\Sigma$. It follows that $i_C(g_k,V) = i_C(h,V)$ where $h(x) =0 $ for all $x\in V$. If $D:=\{0\}$, $h(V)$ is contained in $D$, and $D$ is contained in $V$, so the commutativity property of the fixed point index (see \cite{DG} or \cite{Nus}) implies that $i_C(h,V)=i_D(h,D)$. The normalization property of the fixed point index (see \cite{DG} or \cite{Nus})  implies that $i_D(h,D)$ is the Lefschetz number of the map $h\colon D\to D$, which equals 1. It follows  that $i_C(h,V)=i_D(h,D)=1$, which is a contradiction.
So, $r_C(g_k)>1$ for all $k$ sufficiently large, and hence $\liminf_{k\to\infty} r_C(g_k)\geq 1$. 
\end{proof}
The fact that the maps  $f_k$ in Theorem \ref{thm:2.3} are homogeneous and order-preserving plays a limited role in the proof. Indeed, suppose that $f$, $C$, $V$ and $\lambda$ are as in Theorem \ref{thm:2.3}. For $k\geq 1$ suppose that $f_k\colon \mathrm{cl}(V)\to C$ is continuous and compact, and $\sup\{\|f_k(x)-f(x)\|\colon x\in C \mbox{ and }\|x|=1\}\to 0$ as $k\to\infty$. Then the proof of Theorem \ref{thm:2.3} shows that for all $k$ large, there  exist $x_k\in C$, with $\|x_k\|=1$, and $\lambda_k>0$ such that $f_k(x_k)=\lambda_k x_k$ and $\liminf_{k\to\infty}\lambda_k \geq \lambda$.

A continuous compact order-preserving  homogeneous  map $f\colon C\to C$ on a closed cone $C$ in $X$ is said to have a {\em continuous cone spectral radius} if for each sequence $f_k\colon C\to C$ of continuous compact order-preserving  homogeneous  maps with $\lim_{k\to\infty} \|f-f_k\|_C=0$ we have that 
\[
\lim_{k\to\infty} r_C(f_k) =r_C(f).
\] 
Thus, a combination of Lemma \ref{lem:2.2} and Theorem \ref{thm:2.3} yields the following result. 
 \begin{theorem}\label{thm:2.4} If $f\colon C\to C$ is a continuous compact order-preserving  homogeneous  map and either, $r_C(f)=0$, or, there exist $0<a_1<a_2<a_3<\ldots$ {\bf not} in $\sigma_C(f)$ with $\lim_{k\to\infty} a_k = r_C(f)$, then $f$ has a continuous cone spectral radius. 
 \end{theorem}
In particular we see that every compact linear map $T\colon X\to X$ leaving invariant a closed cone $C$, has a continuous cone spectral radius, as the cone spectrum is totally disconnected. 
 
\section{An example} 
In this section we give an example of a linear map $T\colon C\to C$ on a closed, total cone in a Banach space $X$, which does not have a continuous cone spectral radius. The map $T\colon C\to C$ is compact, but does not have a continuous compact  linear extension as a map from  $X$  to $X$; so, we do not contradict the final remark in the previous section.  
We follow Bonsall \cite[\S 2]{Bon} and define $Y$ to be the Banach space of all continuous functions on $[0,1]$ with $f(0)=0$, and uniform-norm $\|f\|=\max\{|f(t)|\colon   0\leq t\leq 1\}$. Let $K=\{f\in Y\colon f\mbox{ is nonnegative and convex on $[0,1]$}\}$. 
As Bonsall remarks, $K$ is a closed total cone in $Y$, because the subspace $U$ of twice continuously differentiable functions on $[0,1]$ is dense in $Y$ and contained in $K-K$. 

Let $T\colon K\to K$ be defined by 
\[
Tf(t) =f(t/2)\mbox{\quad for all $f\in Y$ and $0\leq t\leq 1$.}
\]
Bonsall \cite[\S 4]{Bon} shows that $T$ is a compact linear map on $K$, but its extension to $Y$ is not compact.  Let us recall the argument that shows that $T$ is compact on $K$. Because $f\in K$ is convex, nonnegative and $f(0)=0$, $f$ is an increasing function on $[0,1]$, so that $\|f\|=f(1)$ for all $f\in K$. Convexity of $f$ implies for $0\leq s<t\leq 1$ that 
\[
0\leq \frac{Tf(t)-Tf(s)}{t/2-s/2} = \frac{f(t/2)-f(s/2)}{t/2-s/2}\leq \frac{f(1)-f(1/2)}{1-1/2}\leq 2f(1).
\]
It  follows that $|Tf(t)-Tf(s)|\leq |t-s|$ for all $s,t\in[0,1]$ and $f\in K$ with 
$\|f\|\leq 1$. This shows that $\{Tf\colon f\in K\mbox{ and }\|f\|\leq 1\}$ is bounded and  equicontinuous, and hence $T$ is compact on $K$. 

To define the approximating maps $T_k\colon K\to K$, we let for $k\geq 3$, 
\[
\phi_k(t) =\left[\begin{array}{ll} t^2 &\mbox{ for }0\leq t\leq (1/2)^k\\
							(1/2)^{2k} +\epsilon_k(t-(1/2)^k) &\mbox{ for } (1/2)^k\leq t\leq 1,\end{array}\right.                                             
\]
where $\epsilon_k = \frac{1}{2}(1-\frac{1}{2^k -1})$. It is easy to verify that $\phi_k(0)=0$, and $\phi_k$ is nonnegative and convex on $[0,1]$. Furthermore $\phi_k(t)\leq t/2$ for $t\in [0,1]$, and $\lim_{k\to\infty} \phi_k(t) =t/2$ uniformly on $[0,1]$. 
For $k\geq 3$ define $T_k\colon K\to K$ by 
\[
T_kf(t) =f(\phi_k(t))\mbox{\quad for }0\leq t\leq 1.
\]
The linear maps $T_k$ have the following properties: 
\begin{enumerate}
\item $T_k$ is continuous and compact on $K$, 
\item $T_k(K)\subseteq K$, and 
\item $\lim_{k\to\infty} \|T-T_k\|_K=0$.
\end{enumerate} 
To show the first assertion we first note that $T_k$ is clearly continuous. Using convexity of $f$ and $\phi_k$ we find for $0\leq s<t\leq 1$ that 
\begin{eqnarray*}
\frac{f(\phi_k(t))-f(\phi_k(s))}{t-s} & = & \Big{(}\frac{f(\phi_k(t))-f(\phi_k(s))}{\phi_k(t)-\phi_k(s)}\Big{)}\Big{(}\frac{\phi_k(t)-\phi_k(s)}{t-s}\Big{)}\\
 & \leq & \Big{(}\frac{f(1)-f(1/2)}{1-1/2}\Big{)}\Big{(}\frac{\phi_k(1)-\phi_k(s)}{1-s}\Big{)}\leq 2\epsilon_k f(1)\leq f(1). 
\end{eqnarray*}
It follows that $|T_kf(t) -T_kf(s)|\leq |t-s|$ for $s,t\in [0,1]$ and $f\in K$ with $\|f\|\leq1$. 
So, $\{T_kf\colon f\in K\mbox{ and }\|f\|\leq 1\}$ is bounded and equicontinuous, which shows that $T_k$ is compact on $K$. 

Note that the second assertion is clear, as each $f\in K$ is convex and increasing on $[0,1]$. To show the third assertion we again use convexity of $f$ to deduce for $0\leq t\leq 1$ that 
\[
\frac{f(t/2)-f(\phi_k(t))}{t/2-\phi_k(t)}\leq \frac{f(1)-f(1/2)}{1-1/2}\leq 2f(1) \leq 2
\]
for all $f\in K$ with $\|f\|\leq 1$, which gives $\lim_{k\to\infty}\|T-T_k\|_K=0$. 

Remark that for $g(t)=t$, then $g\in K$ and $Tg =\frac{1}{2}g$, so that $r_K(T)\geq 1/2$. In fact, Bonsall \cite{Bon} shows that $r_K(T)=1/2$, which is not difficult. 

For integers $m\geq 1$ let $\phi^m_k$ denote the $m$-fold composition of $\phi_k$. Since $\|T_k^m\|_K = T_k^mg(1) =\phi_k^m(1)$, 
\[
r_K(T_k) =\lim_{m\to\infty} \|T_k^m\|_K^{1/m} =\lim_{m\to\infty} \phi_k^m(1)^{1/m}.
\]
As $\phi_k(t)\leq t/2$ for $t\in [0,1]$, we have that $\phi_k^k(1)\leq (1/2)^k$. 
Recall that $\phi_k(t) = t^2$ for all $0\leq t\leq (1/2)^k$. 
So, for $m\geq k$ we have that 
\[
\phi_k^m(1)\leq (1/2)^{2^{m-k}k},
\] 
which implies that 
\[
\lim_{m\to\infty} \|T_k^m\|_K^{1/m}\leq \lim_{m\to\infty} \big{(}(1/2)^{2^{m-k}k}\big{)}^{1/m}=0.
\]
Thus, $\lim_{k\to\infty} r_K(T_k)\neq r_K(T)$. 

Of course this example does not contradict Theorem \ref{thm:2.4}. Simply observe that 
the function $f(t)=t^\alpha\in K$ is an eigenvector of $T$ with eigenvalue $(1/2)^\alpha$ for all $\alpha\geq 1$; so,  $\sigma_K(T)=(0,1/2]$.

\section{Condition $\mathsf{G}$} 
A useful condition in the analysis of the continuity of the cone spectral radius is the so-called condition $\mathsf{G}$, which was introduced in \cite{BNS}. A cone $C$ in $X$ is said to satisfy {\em condition $\mathsf{G}$ at $x\in C$} if for each sequence 
$(x_k)_k$ in $C$ with $\lim_{k\to\infty} x_k=x$ and each $0<\lambda<1$ there exists $m\geq 1$ such that 
\[
\lambda x\leq x_k\mbox{\quad for all }k\geq m.
\]
We say that {\em $C$ satisfies condition $\mathsf{G}$} if condition $\mathsf{G}$ holds at each $x\in C$. 

\begin{theorem}\label{thm:4.1} If $f\colon C\to C$ is a continuous compact homogeneous order-preserving map on a closed  cone $C$ in $X$, and there exists $u\in C$ at which condition $\mathsf{G}$ holds and $f(u)=r_C(f)u$, then $f$ has a continuous cone spectral radius.
\end{theorem}
\begin{proof}
Suppose for $k\geq 1$ that $f_k\colon C\to C$ is a continuous compact homogeneous order-preserving map and (\ref{eq:1.1}) holds. Let $u_k=f_k(u)$ for $k\geq 1$ and note that $\lim_{k\to\infty} u_k  =r_C(f) u$ by (\ref{eq:1.1}). As condition $\mathsf{G}$ holds at $u$, there exists a sequence of positive reals $(\mu_k)_k$, with $\mu_k\to 1^-$ and 
\begin{equation}\label{eq:4.1} 
\mu_k r_C(f)u\leq u_k\mbox{\quad for all }k\geq 1.
\end{equation}
So for each $k\geq 1$, 
\begin{equation}\label{eq:4.2} 
\mu_k r_C(f)u_k= \mu_k r_C(f)f_k(u)\leq f_k(u_k).
\end{equation}
Thus, $\mu_kr_C(f) \leq r_C(f_k)$ for all $k\geq 1$ by Lemma \ref{lem:2.1b}. 
Letting $k\to\infty $ we conclude that 
\[
r_C(f)\leq \liminf_{k\to\infty} r_C(f_k).
\] 
It now follows from Lemma \ref{lem:2.2} that $\lim_{k\to\infty} r_C(f_k) =r_C(f)$.
\end{proof}
It turns out that condition $\mathsf{G}$ always holds at points in the interior of a closed cone. To prove this, it is convenient to recall the definition of Thompson's (part) metric \cite{Tho} on cones and some other related notions. Given a closed cone $C$ in $X$ and $x,y\in C$ we say that $y$ {\em dominates} $x$ if there exists $\beta>0$ such that $x\leq\beta y$. This yields an equivalence relation $\sim_C$ on $C$ by $x\sim_C y$ if $y$ dominates $x$, and $x$ dominates $y$. In other words, $x\sim_C y$ if and only if there exist $0<\alpha\leq \beta$ such that $\alpha y\leq x\leq \beta y$. The equivalence classes are called {\em parts} of the cone. For $x\sim_C y$ one can consider the function  
$M(x/y)=\inf \{\beta>0 \colon x\leq \beta y\}$. Using this function {\em Thompson's (part) metric} is defined by 
\[
d_T(x,y) =\log \max\{M(x/y),M(y/x)\}\mbox{\quad  for } x,y\in C\setminus\{0\}\mbox{ with } x\sim_C y,\] 
and $d_T(0,0)=0$. It is known \cite{Tho} that $d_T$ is a metric on each part of a closed cone $C$. Moreover, if $C$ has a non-empty interior, $C^\circ$, and $x,y\in C^\circ$ with $\|x-y\|<r$, then 
\begin{equation}\label{eq:4.3} 
d_T(x,y) \leq \log \max\Big{\{} \frac{ r+\|x-y\|}{r},\frac{r}{r-\|x-y\|}\Big{\}},
\end{equation}
see \cite[p.16]{Nmem1}.
\begin{lemma}\label{lem:4.2} 
If $C$ is a closed cone in $X$ with a non-empty interior, then $C$ satisfies property $\mathsf{G}$ at each $x\in C^\circ$. 
\end{lemma}
\begin{proof}
Let $(x_k)_k$ be sequence in $C^\circ$ with $\lim_{k\to\infty}\|x_k -x\|=0$ and $0<\lambda<1$. By (\ref{eq:4.3}) there exists $m\geq 1$ such that 
\[
d_T(x_k,x) <\log 1/\lambda\mbox{\quad for all }k\geq m.
\]
This implies that $M(x/x_k)\leq 1/\lambda$, so that $x\leq x_k/\lambda$ for all $k\geq m$. Thus, $\lambda x\leq x_k$ for all $k\geq m$. 
\end{proof}
In the sequel we shall also need the following basic result from \cite{NuLAA}. 
\begin{lemma}\label{lem:4.3} Suppose that $f\colon C\to C$ is an order-preserving homogeneous map  on a closed cone $C$ in $X$ and $x,y\in C$ are such that $y$ dominates $x$ and $x\neq 0$. If $\lambda x\leq f(x)$ and $f(y)\leq \mu y$, then 
$\lambda\leq \mu$. 
\end{lemma}
\begin{proof}
The assertion is obvious if $\lambda =0$. Suppose $\lambda>0$ and note that, as $y$ dominates $x$, there exists $\beta >0$ such that $x\leq \beta y$. This implies that $\lambda x\leq f(x)\leq\beta f(y)\leq \beta\mu y$. It follows that $\lambda^k x\leq f^k(x)\leq \beta f^k(y)\leq \beta \mu^k y$, so that 
\[
\Big{(}\frac{\mu}{\lambda}\Big{)}^k y - \frac{1}{\beta}x\in C\mbox{\quad  for all }k\geq 1.
\]
By letting $k\to \infty$ we find that $-x/\beta\in C$, if $\mu<\lambda$. This is impossible, as $x\neq 0$, and hence $\lambda\leq \mu$.
\end{proof}
Remark that if $x\in C$ and $y\in C^\circ$, then $y$ dominates $x$, as $y-\delta x\in C$ for all $\delta >0$ sufficiently small. Thus, it follows from Lemma \ref{lem:4.3} that if $f\colon C\to C$ is a continuous compact order-preserving homogeneous map on a closed cone $C\subseteq X$ with non-empty interior and $v\in C^\circ$ is an eigenvector of $f$ with eigenvalue $\lambda$, then $\lambda =r_C(f)$. Combining this fact with Theorem \ref{thm:4.1} and Lemma \ref{lem:4.2} yields the following result. 
\begin{corollary}\label{cor:4.4} 
If $f\colon C\to C$ is a continuous compact order-preserving homogeneous map on a closed cone $C$ with non-empty interior, and $f$ has an eigenvector in $C^\circ$, then $f$ has a continuous cone spectral radius. 
\end{corollary}
It was shown \cite{BNS} that if $C$ is a closed cone in a finite dimensional vector space $X$, then $C$ satisfies condition $\mathsf{G}$ if and only if $C$ is polyhedral. Recall that a closed cone $C\subseteq X$ is a {\em polyhedral} cone, if it is intersection of finitely may closed  half-space, i.e., there exists $\phi_1,\ldots,\phi_m\in X^*$ such that 
\[
C = \{x\in X\colon \phi_i(x)\geq 0\mbox{ for all } 1\leq i\leq m\}.
\]
Thus, we have the following consequence of Theorem \ref{thm:4.1}. 
\begin{corollary}\label{cor:4.5} 
If $f\colon C\to C$ is a continuous order-preserving  homogeneous map on a polyhedral cone $C$, then $f$ has a continuous spectral radius.  
\end{corollary}
In the next section we will derive the same result by analyzing the cone spectrum of maps on polyhedral cones and applying Theorem \ref{thm:2.4}. 

\section{The cone spectrum for finite dimensional cones} 
Regarding Theorem \ref{thm:2.4} it is interesting to further analyze the following question:
If $f\colon C\to C$ is a continuous order-preserving homogeneous map on a finite dimensional closed cone $C$ and $r_C(f)>0$, when do there exist $0<a_1<a_2<
a_3<\ldots$ {\bf not} in $\sigma_C(f)$ such that  $\lim_{k\to\infty} a_k = r_C(f)$?
The next example shows that even in finite dimensional spaces such a sequence $(a_k)_k$  may not exist.

\begin{example} Consider the cone of real positive semi-definite $n\times n$ matrices, $\mathrm{Pos}_n(\mathbb{R})$,  inside the vector space of symmetric $n\times n$  matrices. Let $n\geq 2$ and let $A$ be the $n\times n$ diagonal matrix with diagonal  $(1,0,\ldots,0)$. Define $f\colon \mathrm{Pos}_n(\mathbb{R})\to \mathrm{Pos}_n(\mathbb{R})$ by 
\[
f(X) = \Big{(} \mathrm{tr}(XA)X\Big{)}^{1/2}\mbox{\quad for }X\in \mathrm{Pos}_n(\mathbb{R}).
\]
Clearly, if $X\leq Y$, then
$\mathrm{tr}(XA) = \mathrm{tr}(A^TXA)\leq\mathrm{tr}(A^TYA)=\mathrm{tr}(YA)$. It follows from L\"owner's theory \cite{Lo} of order-preserving  maps on $\mathrm{Pos}_n(\mathbb{R})$ that $M\mapsto M^{1/2}$ is order-preserving. Thus, $f$ is a continuous order-preserving  homogeneous map. 

Let $0<\alpha \leq 1$ and let $I$ denote the $(n-2)\times (n-2)$ identity matrix. Consider the $n\times n$ matrix $X_\alpha$ in the interior of  $\mathrm{Pos}_n(\mathbb{R})$,  
\[
X_{\alpha} = \left (\begin{array}{cc|c} 1 &0 & 0\\ 0 & \alpha & 0 \\ \hline 0& 0 & I\\
\end{array} \right). 
\]
Obviously, $\mathrm{tr}(X_{\alpha} A) = 1$, and 
\[
f(X_{\alpha}) =  \left (\begin{array}{cc|c} 1 & 0 & 0\\ 0 & \alpha& 0\\
\hline0 & 0 & I\\  \end{array} \right)^{1/2} =  \left (\begin{array}{cc|c} 1& 0& 0\\ 0 & \sqrt{\alpha} & 0 \\ \hline 0& 0&  I\\ \end{array} \right). 
\]
So, if $f(X_{\alpha}) =\lambda X_{\alpha}$, then $\lambda=1$ and $\sqrt{\alpha}= \alpha$, which is equivalent to $\alpha =1$. Thus, $X_1$ is an eigenvector of $f$ with eigenvalue $1$ in the interior of $\mathrm{Pos}_n(\mathbb{R})$. It follows from Lemma \ref{lem:4.3} that  the cone spectral radius of $f$ is equal to $1$. 

For $0\leq \theta\leq 2\pi$ consider the matrix $Z_\theta$ in the boundary of $\mathrm{Pos}_n(\mathbb{R})$ given by 
\begin{eqnarray*}
Z_\theta & =  & \left (\begin{array}{cc|c} \cos\theta & -\sin\theta & 0 \\ \sin\theta & \cos \theta & 0 \\
\hline 0 & 0 & 0\\ \end{array} \right) \left (\begin{array}{cc|c} 1 & 0 & 0\\ 0 & 0 & 0 \\ \hline 0 & 0& 0\\ \end{array}\right)\left (\begin{array}{cc|c} \cos\theta& \sin\theta & 0\\ -\sin\theta & \cos\theta & 0\\ \hline 0 & 0 & 0\\ \end{array} \right) \\ 
& = & \left (\begin{array}{cc|c} \cos^2\theta & \cos\theta\sin\theta & 0 \\ 
\cos\theta\sin\theta & \sin^2\theta & 0 \\ \hline 0 & 0 & 0\\ \end{array} \right). 
\end{eqnarray*}
Note that 
\[
f(Z_\theta)= \Big{(}\mathrm{tr}(Z_\theta A)Z_\theta \Big{)}^{1/2} = |\cos\theta| Z_\theta^{1/2} = |\cos\theta)| Z_\theta.
\]
So, $Z_\theta$ is an eigenvector with eigenvalue $|\cos\theta|$, and hence  the cone spectrum of $f$ is equal to $[0,1]$. 

Note, however, that as $f$ has an eigenvector in the interior we can apply, instead of Theorem \ref{thm:2.4}, Corollary \ref{cor:4.4} to deduce that $f$ has a continuous cone spectral radius. For $n\geq 3$ the example can  easily be modified so as not to have an eigenvector in the interior, but still have  $[0,1]$ as its cone spectrum. Indeed, let 
$B$ be the $n\times n$ diagonal matrix with diagonal $(1,1,0,\ldots,0)$, and consider $g\colon \mathrm{Pos}_n(\mathbb{R})\to\mathrm{Pos}_n(\mathbb{R})$ given by, 
\[
g(X) = B^T(\mathrm{tr}(XA)X)^{1/2}B\mbox{\quad for }X\in \mathrm{Pos}_n(\mathbb{R}).
\]
The reader can verify that the cone spectrum of $g$ is $[0,1]$ and $g$ has no eigenvectors in the interior of $\mathrm{Pos}_n(\mathbb{R})$. It is, however, unclear if $g$ has a continuous cone spectral radius. 
\end{example}
 
In view of the results so far it interesting to further study the following problem. 
\begin{problem}
Which finite dimensional closed cones $C$ admit  a continuous order-preserving homogeneous map $f\colon C\to C$ with a continuum in its 
cone spectrum? 
\end{problem}
In the remainder we will present some partial results for this problem. We will first recall some basic concepts. 

Given a closed cone $C\subseteq X$ the {\em dual cone} $C^*\subseteq X^*$ is given by $C^*=\{\phi\in X^*\colon \phi(x)\geq 0\mbox{ for all }x\in C\}$. In general $C^*$ need not be a cone, but it is easy to prove that $C^*$ is a cone if $C$ is total. Furthermore, for $x,y\in C$ it follows from the Hahn-Banach separation theorem that 
\begin{equation}\label{HB}
\mbox{$x\leq y$ if and only if $\phi(x)\leq \phi(y)$ for all $\phi\in C^*$.}
\end{equation}
A {\em face} of a closed cone $C\subseteq X$ is a non-empty convex subset $F$ of $C$ such that whenever $x,y\in C$ and $(1-\lambda)x+\lambda y\in F$ for some $0<\lambda <1$ it follows that $x,y\in F$. Note that $C$ and $\{0\}$ are both faces of $C$.  
It is known \cite{lins} that if $C$ is a closed cone in $X$, then the parts of $C$ are precisely the relative interiors of the faces of $C$. 
 
A face $F$ of a polyhedral cone $C$ is called a {\em facet} if $\dim F = \dim C-1$. It is a basic result from polyhedral geometry, see e.g., \cite{Schrijver}, that if $C$ is a polyhedral cone in $X$ with $N$ facets, then there exist $N$ linear functionals $\psi_1,\ldots,\psi_N\in X^*$ such that 
\[
C =\{x\in X\colon \psi_i(x)\geq 0\mbox{ for }i=1,\ldots, N\}\cap \mathrm{span}\, C,
\]
and each linear functional $\psi_i$ corresponds to a unique facet $F_i$ of $C$, in the sense that $F_i = \{x\in C\colon \psi_i(x)= 0\}$. 

Given a closed cone $C$ in $X$ we let $\mathcal{P}(C)$ denote the set of parts of $C$. Note that if $C$ is a polyhedral cone, then $\mathcal{P}(C)$ is finite. On $\mathcal{P}(C)$ we  have a partial ordering $\unlhd$ given by $P\unlhd Q$ if there exist $y\in Q$ and $x\in P$ such that $y$ dominates $x$. Given a polyhedral cone $C$  with facet defining functionals $\psi_1,\ldots,\psi_N$ and $P\in \mathcal{P}(C)$,  we let $I(P)=\{i\colon \psi_i(x)>0\mbox{ for some }x\in P\}$. The next lemma is elementary. 
 \begin{lemma}\label{lem:5.3}
If $C\subseteq V$ is a polyhedral cone with $N$ facets, then 
\begin{enumerate}[(i)]
\item for $P\in\mathcal{P}(C)$ we have $P=\{x\in C\colon \psi_i(x)>0\mbox{ if and only if }i\in I(P)\}$, 
and  
\item $P\unlhd Q$ if and only if $I(P)\subseteq I(Q)$. 
\end{enumerate}
\end{lemma}
\begin{proof}
Let $\psi_1,\ldots,\psi_N\in C^*$ be the facet defining functionals of $C$. By (\ref{HB})  we know that $x\leq y$ is equivalent to $\psi_i(x)\leq\psi_i(y)$ for all $i$. 
For $x\in C$ write $I_x=\{i\colon \psi_i(x)>0\}$, and note that $I_x\subseteq I_y$ if and only if $y$ dominates $x$. Therefore $x\sim_C y$ is equivalent to $I_x=I_y$. 
It follows that $x\in P$ if and only if $I_x=I(P)$, which proves the first assertion. 

If $P\unlhd Q$, there exist $x\in P$ and $y\in Q$ such that $y$ dominates $x$. So, $I(P)=I_x\subseteq I_y=I(Q)$. On the other hand, if $I(P)\subseteq I(Q)$, then for each $x\in P$ and each $y\in Q$ we know that $I_x\subseteq I_y$,  and hence there exists $\beta>0$ such that $\psi_i(x)\leq \beta \psi_i(y)$ for all $i$. It follows from  (\ref{HB})  that $y$ dominates $x$, and hence $P\unlhd Q$. 
\end{proof} 

The following lemma is a direct consequence of Lemma \ref{lem:4.3}. 
\begin{corollary}\label{cor:5.4} 
If $f\colon C\to C$ is an order-preserving homogeneous map on a closed cone 
$C\subseteq V$, and $f$ has eigenvectors $x$ and $y$ in $C$ with $x\sim_C y$, then the eigenvalues of $x$ and $y$ are equal. 
\end{corollary}

As $\{0\}$ is a part of any cone $C$, we see that the number of distinct eigenvalues of an order-preserving homogeneous map $f\colon C\to C$ is bounded by $m-1$, where 
$m$ is the number of parts of $C$. In case $C=\mathbb{R}^n_+$ we can exploit the lattice structure on $\mathbb{R}^n_+$  to construct an example that shows that the upper bound, $2^n-1$, is sharp. Recall that if $x,y\in \mathbb{R}^n_+$, then  $x\vee y\in\mathbb{R}^n_+$ is given by $(x\vee y)_i=\max\{x_i,y_i\}$ for all $i$. 
For  $I\subseteq \{1,\ldots,n\}$ non-empty, define $\chi^I\in\mathbb{R}^n_+$ by $\chi^I_i=1$ if $i\in I$, and $\chi^I_i=0$ otherwise. Furthermore, let 
$\lambda_I>0$ be such that $\lambda_I<\lambda_J$ if $I\subseteq J$ and $I\neq J$. 
Consider the map $f\colon\mathbb{R}^n_+\to\mathbb{R}^n_+$ given by, 
 \[
 f(x)=\bigvee_{\emptyset \neq I\subseteq \{1,\ldots,n\}} \lambda_I (\min_{i\in I} x_i)\chi^I\mbox{\quad\quad  for $x\in\mathbb{R}^n_+$. }
 \]
Obviously, $f$ is a continuous order-preserving homogeneous map on $\mathbb{R}^n_+$, and for $J\subseteq \{1,\ldots,n\}$ with $J\neq \emptyset$ we have that 
 \[
 f(\chi^J)=\bigvee_{\emptyset \neq I\subseteq J} \lambda_I \chi^I = \lambda_J\chi^J,
 \]
 as $\lambda_I<\lambda_J$ for $I\subseteq J$ with $I\neq J$. So, $f$ has $2^n-1$ distinct eigenvalues. 
 
 We will now show that the upper bound, $m-1$, is sharp for every polyhedral cone. 
To establish this result  we need to introduce some more notation. Given a polyhedral cone  $C\subseteq V$  with non-empty interior and facet defining functionals $\psi_1,\ldots,\psi_N$, we define for $r<0$ and $I\subseteq \{1,\ldots,N\}$ non-empty, the function $M_r(I)\colon C^\circ \to [0,\infty)$ by 
 \[
 M_r(I)(x)=\Big{(}\sum_{i\in I} \psi_i(x)^r\Big{)}^{1/r}\mbox{\quad\quad for }x\in C^\circ.
 \]
 From \cite{BNS} we know that $M_r(I)$ has a continuous, order-preserving, homogeneous extension to $\partial C$, as $C$ is polyhedral. Note that 
 if $x\in\partial C$ and there exists $i\in I$ with $\psi_i(x)=0$, then $M_r(I)(x)=0$, since $r<0$. For $r=-\infty$ we define $M_{-\infty}(I)\colon C\to [0,\infty)$ by 
 \[
 M_{-\infty}(I)(x)=\min_{i\in I} \psi_i(x)\mbox{\quad\quad for }x\in C.
 \]
\begin{theorem}\label{thm:5.5} 
Let $C\subseteq V$ be a polyhedral cone with non-empty interior and $m$ faces. If $f\colon C\to C$ is a order-preserving homogeneous  map, then $|\sigma_C(f)|\leq m-1$. Moreover, on $C$ there exists a continuous order-preserving homogeneous  map with $m-1$ distinct eigenvalues.  
\end{theorem}   
\begin{proof}
As the parts of $C$ coincide with the relative interiors of the faces of $C$ we know  that $C$ has $m$ parts, see \cite{lins}.  Omitting the trivial part $\{0\}$ and using Corollary  \ref{cor:5.4} we see that $m-1$ is an upper bound for the size of $\sigma_C(f)$. 
 
To construct an example with $m-1$ distinct eigenvalues on $C$, we let $\psi_1,\ldots,\psi_N$ denote the facet defining functionals of $C$. 
By Lemma \ref{lem:5.3} we know that for each $P\in\mathcal{P}(C)$ we have  that 
\[
P=\{x\in C\colon \psi_i(x)>0\mbox{ if and only if  }i\in I(P)\}.
\]

For each $P\in\mathcal{P}(C)$ with $P\neq \{0\}$ we select $z^P\in P$. 
The idea is to construct a continuous order-preserving homogeneous  map which has the points $z^P$ as eigenvectors with distinct eigenvalues. Let $r\in [-\infty,0)$ and $M_r(I)(x)$ be defined as above for $x\in C$ and $I\subseteq \{1,\ldots,N\}$ non-empty. The continuous order-preserving homogeneous map $f_r\colon C\to C$ will be of the form:
\begin{equation}\label{eq:fr}
f_r(x)=\sum_{P\in\mathcal{P}(C), P\neq\{0\}} \lambda_P M_r(I(P))(x) u^P\mbox{\quad\quad for }x\in C,
\end{equation}
where $\lambda_P>0$ and $u^P\in P$ are chosen appropriately.   

Recall from Lemma \ref{lem:5.3} that $P\unlhd Q$ if and only if $I(P)\subseteq I(Q)$. Thus, for each $x\in Q$ and $P\in\mathcal{P}(C)$ with $P\neq \{0\}$ we have that $M_r(I(P))(x)>0$ if $P\unlhd Q$, and $M_r(I(P))(x)=0$ otherwise. 

We define $\lambda_P>0$ and $u^P\in P$ inductively using the height of $P$ in the finite partially ordered set $(\mathcal{P}(C),\unlhd)$. For $P\in\mathcal{P}(C)$ with height $1$, or equivalently  $\dim P =1$, take $u^P=z^P$ and chose $\lambda_P>0$ such that the positive numbers $\mu_P =\lambda_PM_r(I(P))(u^P)$ are all distinct for $P\in\mathcal{P}(C)$ with height $1$. So, $f(z^P) =\mu_Pz^P$ for those parts $P$.

Suppose that we have already selected $u^P\in P$ and $\lambda_P>0$ for all 
$P\unlhd Q$ with $\{0\}\neq P\neq Q$. Consider 
\[
f_r(z^Q)  =  \sum_{P\in\mathcal{P}(C), P\neq\{0\}} \lambda_P M_r(I(P))(z^Q) u^P
  =  \sum_{P\unlhd Q, P\neq\{0\}} \lambda_P M_r(I(P))(z^Q) u^P
\]
and write 
\[
w^Q = \sum_{P\unlhd Q, \{0\}\neq P\neq Q} \lambda_PM_r(I(P))(z^Q) u^P.
\]
We observe  that for each $\mu_Q>0$ sufficiently large, $\mu_Qz^Q-w^Q\in Q$. Now take $\lambda_Q>0$ and $u^Q\in Q$ such that 
\[
\lambda_QM_r(I(Q))(z^Q) u^Q =\mu_Q z^Q-w^Q. 
\]
Recall that $M_r(I(Q))(z^Q)>0$; so, once we have selected $\mu_Q>0$ and $\lambda_Q>0$ the vector $u^Q\in Q$ is fixed. 

It follows from the construction  that $f_r(z^Q)= w^Q+\mu_Qz^Q -w^Q =\mu_Qz^Q$. 
Thus, we can chose $\mu_Q>0$  such that $f_r$ has $|\mathcal{P}(C)|-1=m-1$ distinct eigenvalues.
 \end{proof} 
We would like to point out  that for $-\infty<r<0$, the function $f$ constructed in the proof of Theorem \ref{thm:5.5} is infinitely differentiable on $C^\circ$. One may wonder whether there exists a continuous order-preserving  homogeneous map on $C$ as in Theorem \ref{thm:5.5} which is continuously differentiable on $C^\circ$ and whose derivative, $Df(\cdot)$,  extends continuously to $0$. The answer is clearly no  if $\dim C\geq 2$, because the limit $L$ of $Df(x)$ as $x\to 0$, would necessarily have $m-1$ distinct eigenvalues. 

Our final result shows that there may be  a countably infinite number 
of distinct points  in the cone spectrum, if the cone is not polyhedral. The proof uses Straszewicz's theorem \cite{Stras}, see also \cite[p.167]{Rock}, which says that the exposed points of a closed convex set $S$ in a finite dimensional vector space $V$ are dense in the extreme points of $S$.
\begin{theorem}\label{thm:5.6}
If $C\subseteq V$ is a closed non-polyhedral cone with non-empty interior, then there exists a continuous order-preserving homogeneous map $f\colon C\to C$ with infinitely many distinct eigenvalues.
\end{theorem}
\begin{proof}
Let $u\in C^\circ$ and note that if $C$ is non-polyhedral, then the dual cone $C^*$ is also non-polyhedral.. Define $\Sigma^*=\{\phi\in C^*\colon \phi(u)=1\}$, which is a compact convex set in $V^*$. As $C^*$ is non-polyhedral, $\Sigma^*$  has infinitely many extreme points. By Straszewicz's theorem \cite{Stras} the exposed points of $\Sigma^*$ are dense in the extreme points of $\Sigma^*$. 
Thus, we can find a sequence $(\phi_k)_k$ in $\Sigma^*$ of distinct exposed points such that $\phi_k\to\psi\in \Sigma^*$ as $k\to\infty$ and $\phi_k\neq \psi$ for all $k\geq 1$. 

As $\phi_k$ is an exposed point of $\Sigma^*$ and $C^{**}=C$, there exists  $x^k\in C$ with $\|x^k\|=1$ such that $\phi_k(x^k)=0$ and $\phi(x^k)>0$ for all $\phi\in \Sigma^*\setminus\{\phi_k\}$. In particular,
\begin{equation}\label{eq:5.6}
\psi(x^k)>0\mbox{\quad and\quad }\phi_m(x^k)>0\mbox{\quad for all }m\neq k.
\end{equation}
Let $f\colon C\to C$ be defined as follows, 
\[
f(x)=\sum_{k=1}^\infty \Big{(}\frac{1}{2}\Big{)}^k \lambda_k \inf_{m\neq k} \phi_m(x)x^k\mbox{\quad\quad for }x\in C,
\]
where $(\lambda_k)_k$ is a bounded sequence of strictly positive reals. For each $k\neq q$, we have that  
\[
\lambda_k\inf_{m\neq k}\phi_m(x^q) =0, 
\]
as $\phi_q(x^q)=0$. 
On the other hand, if $k=q$, then 
\[
\lambda_q\inf_{m\neq q}\phi_m(x^q)>0, 
\]
by (\ref{eq:5.6}) and the fact that $\phi_k\to\psi$ as $k\to\infty$, $\phi_k\neq \phi_q$ for all $k\neq q$, and $\phi_q\neq \psi$.

So, for each $q\geq 1$ we find that 
\[
f(x^q) =\Big{(}\frac{1}{2}\Big{)}^q\lambda_q\inf_{m\neq q}\phi_m(x^q)x^q, 
\]
which shows that $x^q$ is a eigenvector of $f$. By selecting $0<\lambda_q\leq 1$ appropriately we can insure that $f$ has infinitely many distinct eigenvalues. 
\end{proof}


\begin{thebibliography}{1}

\bibitem{Bon} F.F. Bonsall, Linear operators in complete positive cones. 
\emph{Proc. London Math. Soc.} \textbf{8}(3), (1958), 53--75. 

\bibitem{BNS} A.D. Burbanks, R.D. Nussbaum, and C. Sparrow, Extension of order-preserving maps on a cone. \emph{Proc. Royal Soc. Edinburgh Sect. A} \textbf{133}A, (2003), 35--59.


\bibitem{Bur} L. Burlando, Continuity of spectrum and spectral radius in Banach algebras.  In {\em Functional analysis and operator theory (Warsaw, 1992)},  
pp. 53--100, Banach Center Publ., 30, Polish Acad. Sci., Warsaw, 1994.

\bibitem{DG} A. Granas and J. Dugundji, {\em Fixed Point Theory}. Springer Monogr. Math., Springer-Verlag, New York, 2003.

\bibitem{KR}  M.G. Kre\u\i n and M.A. Rutman, 
Linear operators leaving invariant a cone in a Banach space.  
\emph{Amer. Math. Soc. Translation  1950} \textbf{26}, (1950).

\bibitem{lins} B. Lins, {\em  Asymptotic behavior and Denjoy-Wolff theorems for Hilbert metric nonexpansive maps}, Ph.D. thesis, Rutgers University, New Brunswick, USA, 2007.
  
\bibitem{Lo}  K. L\"owner,
\"{U}ber monotone Matrixfunktionen. {\em Math. Z.} {\bf 38}, (1934), 177--216. 

\bibitem{MN1} J. Mallet-Paret and R.D. Nussbaum, Eigenvalues for a class of homogeneous cone maps arising from max-plus operators. 
\emph{Discrete Contin. Dyn. Syst.} \textbf{8} (2002), 519--563.

\bibitem{MN2}  J. Mallet-Paret and R.D. Nussbaum, 
Generalizing the Krein-Rutman theorem, measures of noncompactness and the fixed point index. \emph{J. Fixed Point Theory Appl.} {\bf 7}(1), (2010), 103--143.

\bibitem{Mo1} S. Mouton, On spectral continuity of positive elements.  {\em Studia Math. } {\bf 174}(1),  (2006), 75--84.

\bibitem{Mo2} S. Mouton,  A condition for spectral continuity of positive elements.  {\em Proc. Amer. Math. Soc.}  {\bf 137}(5),  (2009), 1777--1782.

\bibitem{Mur} G.J. Murphy,  Continuity of the spectrum and spectral radius.  {\em Proc. Amer. Math. Soc.} {\bf  82}(4),  (1981), 619--621.
 
\bibitem{New} J.D. Newburgh, The variation of spectra.
{\em Duke Math. J.} {\bf 18}, (1951), 165--176. 

\bibitem{Nu1} R.D. Nussbaum, Eigenvectors of nonlinear positive operators and the linear Krein-Rutman theorem. In E. Fadell and G. Fournier, editors, \emph{Fixed Point Theory}. Lecture Notes in Math. \textbf{886}, 309--331, Springer-Verlag, 1981.

\bibitem{Nus} R.D. Nussbaum, {\em The fixed point index and some applications.}
S\'eminaire de Math\'ematiques Sup\'erieures, 
94. Presses de l'Universit\'e de Montr\'eal, Montreal, QC, 1985. 

\bibitem{NuLAA} R.D.  Nussbaum, Convexity and log convexity for the spectral radius. {\em Linear Algebra Appl.} {\bf 73}, (1986), 59Ð122.

\bibitem{Nmem1} R.D. Nussbaum, 
Hilbert's projective metric and iterated nonlinear maps.
{\em Mem. Amer. Math. Soc.} {\bf 391},(1988), 1--137.

\bibitem{Rick} C.E. Rickart, {\em General theory of Banach algebras}. 
Univ. Ser. in Higher Math., D. van Nostrand Co., Inc., Princeton, N.J., 1960

\bibitem{Rock} R.T. Rockafellar, \emph{Convex Analysis}, Princeton Landmarks in Mathematics, Princeton, N.J., 1997.

\bibitem{Schrijver} A. Schrijver,
\emph{Theory of Linear and Integer Programming}. John Wiley,
Chichester, 1986.


\bibitem{Stras} S. Straszewicz, 
\"{U}ber exponierte Punkte abgeschlossener Punktmengen. {\em Fund. Math.} {\bf 24}, (1935), 139--143. 

\bibitem{Tho} A.C. Thompson, On certain contraction mappings in a partially
ordered vector space. \emph{Proc. Amer. Math. Soc.} \textbf{14}, (1963),
438--443.

\end{thebibliography}
\end{document}